\documentclass[a4paper,12pt]{article}

\usepackage{amsmath,amssymb,amsthm, mathrsfs}
\usepackage{latexsym}
\usepackage{graphics}
\usepackage{hyperref}
\usepackage{soul}
\usepackage{graphicx,psfrag}
\usepackage[bf, small]{titlesec}
\usepackage{authblk} 

\usepackage{lineno}

\theoremstyle{plain}
\newtheorem{Thm}{Theorem}[section]
\newtheorem{Lem}[Thm]{Lemma}
\newtheorem{Prop}[Thm]{Proposition}
\newtheorem{Cor}[Thm]{Corollary}

\theoremstyle{definition}
\newtheorem{Defi}[Thm]{Definition}
\newtheorem{Rem}[Thm]{Remark}

\usepackage{standalone}
\usepackage{pgfpages,tikz,pgfkeys,pgfplots}
\usetikzlibrary{arrows,positioning,matrix,fit,backgrounds,shapes,shapes.geometric, intersections}

\usetikzlibrary{shapes.callouts,decorations.pathmorphing} 
\usetikzlibrary{shadows} 
\usepackage{calc} 
\usetikzlibrary{decorations.markings} 
\tikzstyle{vertex}=[circle, draw, inner sep=0pt, minimum size=6pt] 
\newcommand{\vertex}{\node[vertex]}
\usetikzlibrary{arrows,matrix} 

\newcommand{\DDD}{\mathcal{D}} 

\title{On chordal phylogeny graphs\thanks{This research was supported by
the National Research Foundation of Korea(NRF) funded by the Korea government(MEST) (No.\ NRF-2017R1E1A1A03070489) and by the Korea government(MSIP) (No.\ 2016R1A5A1008055).}}

\date{}
\author[1]{\small Soogang Eoh}
\author[1]{\small Suh-Ryung Kim}
\affil[1]{\footnotesize Department of Mathematics Education, Seoul National University, Seoul 08826}
\affil[ ]{\footnotesize\textit{mathfish@snu.ac.kr, srkim@snu.ac.kr}}

\begin{document}
\maketitle
\begin{abstract}
An acyclic digraph each vertex of which has indegree at most $i$ and outdegree at most $j$ is called an $(i, j)$ digraph for some positive integers $i$ and $j$.
Lee {\it et al.} (2017) studied the phylogeny graphs of $(2, 2)$ digraphs and gave sufficient
conditions and necessary conditions for $(2, 2)$ digraphs having chordal phylogeny graphs.
Their work was motivated by problems related to evidence propagation in a Bayesian network for which it is useful to know which acyclic digraphs have their moral graphs being chordal (phylogeny graphs are called moral graphs in Bayesian network theory).

In this paper, we extend their work.  We  completely characterize phylogeny graphs of $(1, i)$ digraphs and $(i,1)$ digraphs, respectively, for a positive integer $i$.
Then, we study phylogeny graphs of a $(2,j)$ digraphs, which  is worthwhile in the context that a child has two biological parents in most species, to show that the phylogeny graph of a $(2,j)$ digraph $D$ is chordal if the underlying graph of $D$ is chordal for any positive integer $j$.
Especially, we show that as long as the underlying graph of a $(2,2)$ digraph is chordal, its phylogeny graph is not only chordal but also planar.
\end{abstract}
\noindent
{\bf Keywords:} competition graph; moral graph; phylogeny graph; $(i, j)$ digraph; chordal graph; planar graph

\noindent
{\bf 2010 Mathematics Subject Classification:} 05C20, 05C75, 94C15

\section{Introduction}
Throughout this paper, we deal with simple digraphs.

Given an acyclic digraph $D$, the \emph{competition graph} of $D$, denoted by $C(D)$, is the simple graph having vertex set $V(D)$ and edge set $\{uv \mid (u, w), (v, w) \in A(D) \text{ for some } w \in V(D) \}$.
Since Cohen~\cite{cohen1968interval} introduced the notion of competition graphs in the study on predator-prey concepts in ecological food webs, various variants of competition graphs have been introduced and studied.

The notion of phylogeny graphs was introduced by Roberts and Sheng~\cite{roberts1997phylogeny} as a variant of competition graphs. Given an acyclic digraph $D$, the \emph{underlying graph} of $D$, denoted by $U(D)$, is the simple graph with vertex set $V(D)$ and edge set $\{uv \mid (u, v) \in A(D) \text{ or }(v, u) \in A(D) \}$.
The \emph{phylogeny graph} of an acyclic digraph $D$, denoted by $P(D)$, is the graph with the vertex set $V(D)$ and edge set $E(U(D)) \cup E(C(D))$.

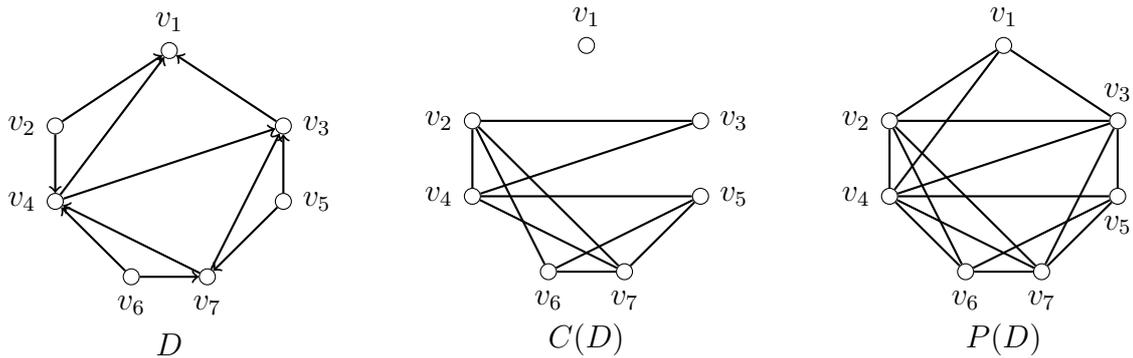
\begin{figure}
\begin{center}
\begin{tikzpicture}[x=1.0cm, y=1.0cm]

    \vertex (b1) at (1.5,3) [label=above:$v_1$]{};
    \vertex (b2) at (0,2) [label=left:$v_2$]{};
    \vertex (b3) at (3,2) [label=right:$v_3$]{};

    \vertex (b4) at (0,1) [label=left:$v_4$]{};
    \vertex (b5) at (3,1) [label=right:$v_5$]{};
    \vertex (b6) at (1,0) [label=below:$v_6$]{};
    \vertex (b7) at (2,0) [label=below:$v_7$]{};

    \path
 (b2) edge [->,thick] (b1)
 (b2) edge [->,thick] (b4)
 (b5) edge [->,thick] (b3)
 (b5) edge [->,thick] (b7)
 (b6) edge [->,thick] (b4)
 (b6) edge [->,thick] (b7)
 (b7) edge [->,,thick] (b3)
 (b7) edge [->,thick] (b4)
 (b4) edge [->,thick] (b1)
 (b4) edge [->,thick] (b3)
 (b3) edge [->,thick] (b1)
;
 \draw (1.5,-0.9) node{$D$};
\end{tikzpicture}
\qquad
\begin{tikzpicture}[x=1.0cm, y=1.0cm]
\vertex (b1) at (1.5,3) [label=above:$v_1$]{};
    \vertex (b2) at (0,2) [label=left:$v_2$]{};
    \vertex (b3) at (3,2) [label=right:$v_3$]{};

    \vertex (b4) at (0,1) [label=left:$v_4$]{};
    \vertex (b5) at (3,1) [label=right:$v_5$]{};
    \vertex (b6) at (1,0) [label=below:$v_6$]{};
    \vertex (b7) at (2,0) [label=below:$v_7$]{};

    \path
 (b2) edge [-,thick] (b4)
 (b4) edge [-,thick] (b3)
 (b2) edge [-,thick] (b3)
 (b4) edge [-,thick] (b5)
 (b5) edge [-,thick] (b7)
 (b7) edge [-,thick] (b4)
 (b2) edge [-,thick] (b6)
 (b6) edge [-,thick] (b7)
 (b7) edge [-,thick] (b2)
 (b5) edge [-,thick] (b6);

 \draw (1.5,-0.9) node{$C(D)$}
	;
\end{tikzpicture}
%
%
%
%
%
%
\qquad
\begin{tikzpicture}[x=1.0cm, y=1.0cm]

 \vertex (b1) at (1.5,3) [label=above:$v_1$]{};
    \vertex (b2) at (0,2) [label=left:$v_2$]{};
    \vertex (b3) at (3,2) [label=above:$v_3$]{};

    \vertex (b4) at (0,1) [label=left:$v_4$]{};
    \vertex (b5) at (3,1) [label=below:$v_5$]{};
    \vertex (b6) at (1,0) [label=below:$v_6$]{};
    \vertex (b7) at (2,0) [label=below:$v_7$]{};

    \path
 (b2) edge [-,thick] (b1)
 (b2) edge [-,thick] (b4)
 (b5) edge [-,thick] (b3)
 (b5) edge [-,thick] (b7)
 (b6) edge [-,thick] (b4)
 (b6) edge [-,thick] (b7)
 (b7) edge [-,,thick] (b3)
 (b7) edge [-,thick] (b4)
 (b4) edge [-,thick] (b1)
 (b4) edge [-,thick] (b3)
 (b3) edge [-,thick] (b1)
 (b2) edge [-,thick] (b3) 
 (b2) edge [-,thick] (b6) 
 (b2) edge [-,thick] (b7) 
 (b4) edge [-,thick] (b5) 
 (b5) edge [-,thick] (b6) 
	;

 \draw (1.5,-0.9) node{$P(D)$}
	;
\end{tikzpicture}
\end{center}
\caption{An acyclic digraph $D$,
 the competition graph $C(D)$ of $D$, and
the phylogeny graph $P(D)$ of $D$.}
\label{fig:counterexample}
\end{figure}

``Moral graphs'' having arisen from studying Bayesian networks are the same as phylogeny graphs.
One of the best-known problems, in the context of Bayesian networks, is related to the propagation of evidence.
It consists of the assignment of probabilities to the values of the rest of the variables, once the values of some variables are known. Cooper~\cite{cooper1990computational} showed that this problem is NP-hard.
Most noteworthy algorithms for this problem are given by Pearl~\cite{pearl1986fusion}, Shachter~\cite{shachter1988probabilistic} and by Lauritzen and Spiegelhalter~\cite{lauritzen1988local}. Those algorithms include a step of triangulating a moral graph, that is, adding edges to a moral graph to form a chordal graph.

A \emph{hole} of a graph is an induced cycle of length at least four.
A graph is said to be \emph{chordal} if it does not contain a hole.
As triangulations of moral graphs play an important role in algorithms for propagation of evidence in a Bayesian network, studying chordality of the phylogeny graphs of acyclic digraphs is meaningful.

For positive integers $i$ and $j$, an \emph{$(i, j)$ digraph} is an acyclic digraph such that each vertex has indegree at most $i$ and outdegree at most $j$.
A graph $G$ is a \emph{phylogeny graph} (resp.\ an \emph{$(i, j)$ phylogeny graph}) if there is an acyclic digraph $D$ (resp.\ an $(i, j)$ digraph $D$) such that $P(D)$ is isomorphic to $G$.
Throughout this paper, we assume that variables $i$ and $j$ belong to the set of positive integers unless otherwise stated.

In this paper, we completely characterize the phylogeny graphs of $(1, j)$ digraphs and those of $(i, 1)$ digraphs (Theorem~\ref{thm:(1,j) realizable} and ~\ref{thm:(i,1) realizable}).
Then we study the phylogeny graphs of $(2, j)$ digraphs.
We show that the phylogeny graph of any $(2, j)$ digraph whose underlying graph is chordal is chordal (Theorem~\ref{thm:(2,j) chordal}).
Finally, we show that the phylogeny graph of any $(2, 2)$ digraph whose underlying graph is chordal is chordal and planar (Theorem~\ref{thm:planar}).

\section{$(1, j)$ phylogeny graphs and $(i, 1)$  phylogeny graphs}

In this section, we characterize the $(1, j)$ phylogeny graphs and the $(i, 1)$ phylogeny graphs.

A {\em component} of a digraph $D$ is the subdigraph of $D$ induced by the vertex of a component of its underlying graph.
Given an acyclic digraph $D$, it is easy to check that $D'$ is a component of $D$ if and only if $P(D')$ is a component of $P(D)$.
Thus,
\begin{itemize}
  \item[($\star$)] it is sufficient to consider only weakly connected digraphs (whose underlying graphs are connected) in studying phylogeny graphs of digraphs. \end{itemize}

First we take care of $(1, j)$ phylogeny graphs.

A vertex of degree one is called a \emph{pendant vertex}.

Given a graph $G$ and a vertex $v$ of $G$, we denote the set of neighbors of $v$ in $G$ by $N_G(v)$.
We call $N_G(v) \cup \{v\}$ the \emph{closed neighborhood} of $v$ and denote it by $N_G[v]$.
We call $\Delta(G):=\max\{|N_G(v)| \mid v \in V(G)\}$ the \emph{maximum degree of $G$}.

\begin{Thm}\label{thm:(1,j) realizable}
For a positive integer $j$, a graph is a $(1, j)$ phylogeny graph  if and only if it is a forest with the maximum degree at most $j+1$.
\end{Thm}
\begin{proof}
By ($\star$), it suffices to show that, for a positive integer $j$, a connected graph is a $(1, j)$ phylogeny graph if and only if it is a tree with the maximum degree at most $j+1$.
To show the ``only if'' part, suppose that a connected graph $G$ is a $(1, j)$ phylogeny graph for some positive integer $j$.
Then there is a $(1, j)$ digraph $D$ such that $P(D)$ is isomorphic to $G$.
Since every vertex of $D$ has indegree at most one, $P(D)=U(D)$.
Since $P(D)$ is connected, $U(D)$ is connected.
Moreover, since $D$ is a $(1, j)$ digraph, $U(D)$ has the maximum degree at most $j+1$.
If $U(D)$ contained a cycle $C$, then there would exist a vertex on $C$ of indegree at least two by the acyclicity of $D$.
Therefore $U(D)$ does not contain a cycle, and so $U(D)$ is tree.

Now we show the ``if'' part.
If $T$ is a tree with one or two vertices, then it is obviously a $(1, 1)$ phylogeny graph.
We take a tree $T$ with at least three vertices and let $j=\Delta(T)-1$.
Then there exist pendant vertices. We take one of them and denote it by $u$.
We regard $T$ as a rooted tree with the root $u$ and define an oriented tree $\overrightarrow{T}$, which is acyclic, with $V(\overrightarrow{T})=V(T)$ as follows.
We take an edge $xy$ in $T$.
Then $d_T(u, x)=d_T(u, y)+1$ or $d_T(u, y)=d_T(u, x)+1$.
If the former, $(y, x) \in A(\overrightarrow{T})$ and if the latter, $(x, y) \in A(\overrightarrow{T})$.
By definition, $U(\overrightarrow{T})=T$.
Moreover, $u$ has indegree zero and outdegree one, and each vertex in $\overrightarrow{T}$ except $u$ has indegree one in $\overrightarrow{T}$.
Then, since the degree of each vertex in $T$ is at most $j+1$, the outdegree of each vertex in $\overrightarrow{T}$ is at most $j$.
Therefore $\overrightarrow{T}$ is a $(1, j)$ digraph.
Since each vertex in $\overrightarrow{T}$ has indegree at most one, $P(\overrightarrow{T})=U(\overrightarrow{T})=T$.
\end{proof}

\noindent
If $P(D)$ is triangle-free for an acyclic digraph $D$, then the indegree of each vertex is at most one in $D$, for otherwise, the vertex with indegree at least two form a triangle with two in-neighbors in $P(D)$.
Thus, the following corollary immediately follows from Theorem~\ref{thm:(1,j) realizable}.

\begin{Cor}
For any positive integers $i$ and $j$, if an $(i, j)$ phylogeny graph is triangle-free, then it is a forest with the maximum degree at most $j+1$.
\end{Cor}

A \emph{clique} of a graph $G$ is a set $X$ of vertices of $G$ with the property that every pair of distinct vertices in $X$ are adjacent in $G$. A \emph{maximal clique} of a graph $G$ is a clique $X$ of vertices of $G$, such that there is no clique of $G$ that is a proper superset of $X$.
The size of a maximum clique of a graph $G$ is called a \emph{clique number} and denoted by $\omega(G)$.

Given a digraph $D$ with $n$ vertices, a one-to-one correspondence $f:V(D) \to [n]$ is called an \emph{acyclic labeling} of $D$ if $f(u) > f(v)$ for any arc $(u, v)$ in $D$.
It is well-known that $D$ is acyclic if and only if there is an acyclic labeling of $D$.

Given a digraph $D$ and a vertex $v$ of $D$, we denote the set of in-neighbors of $v$ in $D$ by $N_D^-(v)$.
We call $N_D^-(v) \cup \{v\}$ the \emph{closed in-neighborhood} of $v$ and denote it by $N_D^-[v]$.
In addition, for a vertex $v$ in $D$, we call $|N_D^+(v)|$ and $|N_D^-(v)|$ the \emph{outdegree} and the \emph{indegree} of $v$, respectively, and denote it by $d_D^+(v)$ and $d_D^-(v)$, respectively.

Given a graph $G$, a vertex $v$ of $G$ is called a \emph{simplicial vertex} if $N_G[v]$ forms a clique in $G$.

Now we consider phylogeny graphs of $(i, 1)$ digraphs.

\begin{Lem}\label{lem:acyclic labeling}
Let $D$ be a nontrivial weakly connected $(i, 1)$ digraph for some positive integer $i$ and $f$ be an acyclic labeling of $D$.
Then every maximal clique in $P(D)$ is in the form of the closed in-neighborhood of the vertex with the least $f$-value among the vertices in the maximal clique.
%
\end{Lem}
\begin{proof}
Let $X$ be a maximal clique in $P(D)$ and $x$ be the vertex having the least $f$-value among the vertices in $X$.
Suppose $X \not\subset N_D^-[x]$.
Then there is a vertex $y \in X$ such that $(y, x) \notin A(D)$.
Since $x$ has the least $f$-value among the vertices in $X$, $(x, y) \notin A(D)$.
Yet, since $x$ and $y$ are adjacent in $P(D)$, they have a common out-neighbor, say $z$, in $D$.
By the hypothesis that $D$ is an $(i, 1)$ digraph, $z$ is the only out-neighbor of $x$ and $y$.
Since $x$ has the least $f$-value among the vertices in $X$, $z \notin X$.
Since $N_D^-[z]$ forms a clique in $P(D)$, $X \not\subset N_D^-[z]$ by the maximality of $X$.
That is, there exists a vertex $w$ in $X$ but not in $N_D^-[z]$.
Then $w \neq z$.
Since $z$ is the unique out-neighbor of $x$ and $y$ in $D$, $(x, w) \notin A(D)$ and $(y, w) \notin A(D)$.
Furthermore, since $w \notin N_D^-[z]$, neither $w$ and $x$ nor $w$ and $y$ have a common out-neighbor in $D$.
However, $w$, $x$, and $y$ belong to $X$, so $(w, x)$ and $(w, y)$ are arcs in $D$, which is a contradiction to the hypothesis that $D$ is a $(i, 1)$ digraph.
Hence $X \subset N_D^-[x]$.
Since $N_D^-[x]$ is a clique in $P(D)$, $X = N_D^-[x]$ by the maximality of $X$.
\end{proof}

\begin{Lem}\label{lem:maximalclique}
Given a nontrivial weakly connected $(i, 1)$ digraph $D$ for a positive integer $i$, the set of all the maximal cliques in $P(D)$ is exactly the set
\[
\{N_D^-[u] \mid u \in V(D) \text{ and } d_D^-(u) \ge 1\}.
\]
\end{Lem}
\begin{proof}
Let $f$ be an acyclic labeling of $D$.
Take a maximal clique $Y$ in $P(D)$.
By Lemma~\ref{lem:acyclic labeling}, $Y=N_D^-[y]$ for the vertex $y$ having the least $f$-value among the vertices in $Y$.
Since $D$ is nontrivial and weakly connected, $|Y| \ge 2$ and so $y$ has an in-neighbor in $D$, i.e.\ $d_D^-(y) \ge 1$.

To prove a containment in the other direction, take a vertex $u$ of indegree at least one in $D$.
Let $v$ be an in-neighbor of $u$ in $D$.
Suppose, to the contrary, $N_D^-[u]$ is not maximal.
Then there is a maximal clique $X$ properly containing $N_D^-[u]$.
By Lemma~\ref{lem:acyclic labeling}, $X=N_D^-[x]$ for the vertex $x$ with the least $f$-value among the vertices in $X$.
Since $N_D^-[u]$ is properly contained in $N_D^-[x]$, $u \neq x$.
In addition, since $N_D^-[u]$ is included in $N_D^-[x]$, $v$ is also an in-neighbor of $x$ in $D$.
Thus the outdegree of $v$ is at least two, which contradicts the fact that $D$ is an $(i, 1)$ digraph.
Therefore $N_D^-[u]$ forms a maximal clique in $P(D)$ and this completes the proof.
\end{proof}

A \emph{diamond} is a graph obtained from $K_4$ by deleting an edge.
A graph is called \emph{diamond-free} if it does not contain a diamond as an induced subgraph.

\begin{Lem}\label{lem:diamond-free chordal}
The phylogeny graph of a weakly connected $(i, 1)$ digraph for a positive integer $i$ is diamond-free and chordal.
\end{Lem}
\begin{proof}
Let $D$ be a weakly connected $(i, 1)$ digraph for a positive integer $i$.
We prove the lemma statement by induction on $|V(D)|$.
If $|V(D)| =1$ or $2$, then the statement is trivially true.
Suppose that $|V(D)|=n+1$ and the lemma statement is true for any weakly connected $(i, 1)$ digraph with $n$ vertices ($n \ge 2$).
Since $D$ is acyclic, there is a vertex $u$ of indegree zero in $D$.
Since $D$ is a weakly connected $(i, 1)$ digraph, $d_D^+(u)=1$.
Thus there is a unique out-neighbor $v$ of $u$ in $D$.
Then, as $u$ has indegree of zero in $D$, we may conclude that, for a vertex $w$ in $D$,  $u$ is adjacent to $w$ in $P(D)$ if and only if $w=v$ or $w$ is an in-neighbor of $v$ in $D$, i.e.\
\begin{equation}\label{eqn:D-u}
N_{P(D)}[u] = N_D^-[v].
\end{equation}
Since the indegree and the outdegree of $u$ are zero and one, respectively, $D-u$ is weakly connected.
Obviously $D-u$ is an $(i, 1)$ digraph.
Thus, by the induction hypothesis, $P(D-u)$ is diamond-free and chordal.
Take two vertices $x$ and $y$ in $V(D) \setminus \{u\}$.
Since $u$ has indegree zero, $u$ cannot be a common out-neighbor of $x$ and $y$. Therefore, $x$ and $y$ are adjacent in $P(D)-u$ if and only if $(x, y) \in A(D)$ or $(y, x) \in A(D)$ or they have a common out-neighbor other than $u$ in $D$  if and only if  $x$ and $y$ are adjacent in $P(D-u)$.
Thus we have shown that $P(D)-u = P(D-u)$.
By \eqref{eqn:D-u}, $u$ is simplicial in $P(D)$, so $P(D)$ is chordal. Now it remains to show that $P(D)$ is diamond-free.

Suppose that $P(D)$ has a diamond. Then, since $P(D)-u$ is diamond-free, every diamond of $P(D)$ contains $u$ and a vertex which is not adjacent to $u$ in $P(D)$.
Let $z$ be a vertex on a diamond which is not adjacent to $u$ in $P(D)$.
Then $z$ is not contained in $N_D^-[v] \setminus \{u\}$ and is adjacent to two vertices $y_1$ and $y_2$ in $N_D^-[v] \setminus \{u\}$ by \eqref{eqn:D-u}.
Since $P(D)-u=P(D-u)$, $z$ is adjacent to $y_1$ and $y_2$ in $P(D-u)$.
Moreover, $u$ is not a pendant vertex, so $v$ has an in-neighbor distinct from $u$ in $D$.
Then $v$ has indegree at least one in $D-u$, so $N_{D-u}^-[v]$ is a maximal clique in $P(D-u)$ by Lemma~\ref{lem:maximalclique}.
Obviously $N_{D-u}^-[v]=N_D^-[v] \setminus \{u\}$, so $N_D^-[v] \setminus \{u\}$ is a maximal clique in $P(D-u)$.
Then, since $z$ belongs to $P(D-u)$ and is not contained in $N_D^-[v] \setminus \{u\}$, there exist a vertex $w$ in $N^-_{D-u}[v]$ which is not adjacent to $z$ in $P(D-u)$.
Then the subgraph induced by $z$, $w$, $y_1$, and $y_2$ is a diamond in $P(D-u)$ and we have reached a contradiction.
\end{proof}

\begin{Lem}\label{lem:intersection}
Let $D$ be an $(i, 1)$ digraph for a positive integer $i$ and $f$ be an acyclic labeling of $D$.
Suppose that non-disjoint vertex sets $X$ and $Y$ form distinct maximal cliques in $P(D)$, respectively. Then $X$ and $Y$ have exactly one common vertex, namely $v$, and \[f(v)= \min\{f(w) \mid w \in X \} \text{ or }  \min\{f(w) \mid w \in Y\}\]
whereas \[f(v) > \min\{f(w) \mid w \in X \cup Y\}.\]
\end{Lem}
\begin{proof}
By ($\star$), we may assume that $U(D)$ is connected.
By Lemma~\ref{lem:diamond-free chordal},
$P(D)$ is diamond-free, so  $|X \cap Y|\le 1$.
Then, by the hypothesis that $X$ and $Y$ are non-disjoint vertex sets, $|X \cap Y|=1$.
Let $v$ be the vertex common to $X$ and $Y$.
Since $X$ and $Y$ form maximal cliques, $X=N_D^-[x]$ and $Y=N_D^-[y]$ for the vertices $x$ and $y$ with the smallest $f$-values among the vertices in $X$ and the vertices in $Y$, respectively, by Lemma~\ref{lem:acyclic labeling}.
Since $X$ and $Y$ are distinct, $x \neq y$.
If $v \notin \{x, y\}$, then $x$ and $y$ are two distinct out-neighbors of $v$, which is impossible.
Thus $v \in \{x, y\}$.
Without loss of generality, we may assume $v=x$.
Since $x$ and $y$ have the the smallest $f$-values among the vertices in $X$ and the vertices in $Y$, respectively, $v$ has the least $f$-value among the vertices in $X$ but not among the vertices in $Y$, and the lemma statement is true.
\end{proof}

We shall completely characterize the $(i, 1)$ phylogeny graphs in terms of ``clique graph'' which was introduced by Hamelink~\cite{hamelink1968partial}.

\begin{Defi}
The \emph{clique graph} of a graph $G$, denoted by $K(G)$, is a simple graph such that
\begin{itemize}
  \item every vertex of $K(G)$ represents a maximal clique of $G$;
  \item two vertices of $K(G)$ are adjacent when they share at least one vertex in common in $G$.
\end{itemize}
\end{Defi}

\begin{Thm}\label{thm:(i,1) realizable}
For some positive integer $i$, a graph $G$ is an $(i, 1)$ phylogeny graph  if and only if it is a diamond-free chordal graph with $\omega(G) \le i+1$ and its clique graph is a forest.
\end{Thm}
\begin{proof}
By ($\star$), it is sufficient to show that a connected graph $G$ is an $(i, 1)$ phylogeny graph for some positive integer $i$ if and only if it is a diamond-free chordal graph with $\omega(G) \le i+1$ and its clique graph is a tree.
To show the ``only if'' part, suppose that a connected graph $G$ is an $(i, 1)$ phylogeny graph for some positive integer $i$.
Then $G=P(D)$ for some weakly connected $(i, 1)$ digraph $D$.
By Lemma~\ref{lem:acyclic labeling}, $\omega(G) \le i+1$.
In addition, by Lemma~\ref{lem:diamond-free chordal}, $P(D)$ is diamond-free and chordal.
Now we show that the clique graph $K(G)$ is a tree.
As the clique graph of a connected graph is connected, it is sufficient to show that $K(G)$ is acyclic.
Suppose, to the contrary, that $K(G)$ contains a cycle $C:=X_1X_2 \cdots X_rX_1$ for an integer $r \ge 3$ and maximal cliques $X_1$, $\ldots$, $X_r$ of $G$.
Let $f$ be an acyclic labeling of $D$.
We denote by $x_i$ the vertex which has the least $f$-value in $X_i$ for each $i=1$, $2$, $\ldots$, $r$.
By Lemma~\ref{lem:intersection}, $X_1 \cap X_2 = \{x_1\}$ or $X_1 \cap X_2 = \{x_2\}$.
Without loss of generality, we may assume that  $X_1 \cap X_2 = \{x_2\}$ so that $f(x_1) < f(x_2)$.
By Lemma~\ref{lem:intersection} again, $X_2 \cap X_3 = \{x_2\}$ or $X_2 \cap X_3 = \{x_3\}$.
Suppose that $X_2 \cap X_3 = \{x_2\}$. Then $f(x_3) < f(x_2)$ and $x_2 \in X_1 \cap X_3$.
By Lemma~\ref{lem:intersection}, $X_1 \cap X_3 = \{x_2\}$, and either $f(x_2) = f(x_1)$ or $f(x_2) = f(x_3)$, which contradicts the fact that $f(x_1) < f(x_2)$ and $f(x_3) < f(x_2)$.
Thus $X_2 \cap X_3 = \{x_3\}$.
Continuing in this way, we may show that $X_i \cap X_{i+1} = \{x_{i+1}\}$ for each $i \in \{1, 2, \ldots, r-1\}$ and $X_r \cap X_1 = \{x_1\}$.
By Lemma~\ref{lem:acyclic labeling}, $X_i = N_D^-[x_i]$ for each $i \in \{1, 2, \ldots, r\}$.
Therefore $(x_1, x_r) \in A(D)$ and  $(x_{i+1}, x_i) \in A(D)$ for each $i \in \{1, 2, \ldots, r-1\}$.
Thus $x_1 \rightarrow x_r \rightarrow \cdots \rightarrow x_2 \rightarrow x_1$ is a directed cycle in $D$ and we reach a contradiction to the acyclicity of $D$.
Hence $K(G)$ does not contain a cycle and so the ``only if'' part is true.

To show the ``if'' part, suppose that a connected graph $G$ is diamond-free and chordal with $\omega(G) \le i+1$ and that $K(G)$ is a tree for some positive integer $i$.
If $G$ is a complete graph, then it has at most $i+1$ vertices and is obviously an $(i, 1)$ phylogeny graph.
Thus we may assume that $G$ is not a complete graph.
Then $K(G)$ is not a trivial tree.
We show by induction on $|V(G)|$ that $G$ is an $(i,1)$ phylogeny graph.
Since $G$ is connected and not complete, $|V(G)| \ge 3$.
If $|V(G)|=3$, then $G$ is a path of length two and, by Theorem~\ref{thm:(1,j) realizable}, a $(1, 1)$ phylogeny graph.
Assume that a connected non-complete graph is an $(i, 1)$ phylogeny graph if it is a diamond-free chordal graph with less than $n$ vertices and the cliques of size at most $i+1$ and its clique graph is a tree for $n \ge 4$.
Suppose that $|V(G)|=n$.
Since $K(G)$ is not a trivial tree, it contains a pendant vertex.
Let $X$ be a pendant vertex and $Y$ be the neighbor of $X$ in $K(G)$.
Then $|X \cap Y| \ge 1$. Since $G$ is a connected graph with at least four  vertices, by the maximality of $X$ and $Y$,  $2 \le |X|$ and $2 \le |Y|$.
By Lemma~\ref{lem:intersection}, $X \cap Y=\{u\}$  for some vertex $u$.
Since $|X| \ge 2$, there exists a vertex $v$ in $X \setminus \{u\}$.
Since $K(G)$ does not contain a triangle, $X$ and $Y$ are the only maximal cliques that contain $u$ in $G$ and so
\begin{itemize}
\item[($\dag$)]  $G-v$ does not have a maximal clique containing $u$ other than $X \setminus \{v\}$ (not necessarily maximal) and $Y$.
\end{itemize}
Furthermore, since $X$ is a pendant vertex in $K(G)$, every vertex in $X \setminus \{u\}$ is a simplicial vertex in $G$ and therefore $v$ is a simplicial vertex of $G$.
Then the closed neighborhood of $v$ in $G$ is $X$. Moreover, it is obvious that $G-v$ is a connected diamond-free chordal graph with $\omega(G-v) \le i+1$ and $K(G-v)$ is a tree.
Therefore, by the induction hypothesis, $G-v$ is an $(i, 1)$ phylogeny graph.
Thus there is an $(i, 1)$ digraph $D^*$ such that $P(D^*) = G-v$.
Let $f^*$ be an acyclic labeling of $D^*$.

{\it Case 1}. The vertex $u$ has the least $f^*$-value in $Y$. Then, by Lemma~\ref{lem:acyclic labeling}, $Y=N^-_{D^*}[u]$. Consider the case in which $u$ has no out-neighbor in $D^*$.
Then, by Lemma~\ref{lem:intersection}, $X\setminus\{v\} =\{u\}$.
Adding the vertex $v$ and the arc $(u, v)$ to $D^*$ results in an $(i, 1)$ digraph whose phylogeny graph is $G$.
Now consider the case in which $u$ has an out-neighbor $w$ in $D^*$.
Then $f^*(w) < f^*(u)$ and $d_{D^*}^-(w) \ge 1$.
Since  $d_{D^*}^-(w) \ge 1$, $N_{D^*}^-[w]$ forms a maximal clique by Lemma~\ref{lem:maximalclique}.
Since  $f^*(w) < f^*(u)$ and  $f^*(u)$ is the minimum in $Y$, $N_{D^*}^-[w]$ is distinct from $Y$.
Since $N_{D^*}^-[w]$ contains $u$, $N_{D^*}^-[w]=X\setminus\{v\}$ by ($\dag$).
Since $|X| \le i+1$, $|X\setminus\{v\}| \le i$ and so
$d_{D^*}^-(w) \le i-1$.
Adding the vertex $v$ and the arc $(v, w)$ to $D^*$ results in an $(i, 1)$ digraph whose phylogeny graph is $G$.

{\it Case 2}. The vertex $u$ does not have the least $f^*$-value in $Y$.
Then $u$ has the least $f^*$-value in $X \setminus \{v\}$ by  Lemma~\ref{lem:intersection}.
Thus, if $u$ has no in-neighbor in $D^*$, then $X\setminus\{v\} =\{u\}$, and so adding the vertex $v$ and the arc $(v, u)$ to $D^*$ results in an $(i, 1)$ digraph whose phylogeny graph is $G$.
Now consider the case in which $u$ has an in-neighbor $w$ in $D^*$.
Then $d_{D^*}^-(u) \ge 1$, so $N_{D^*}^-[u]$ forms a maximal clique by Lemma~\ref{lem:maximalclique}.
By Lemma~\ref{lem:acyclic labeling}, $u$ has the least $f^*$-value in $N_{D^*}^-[u]$.
Since $u$ does not have the least $f^*$-value in $Y$, $N_{D^*}^-[u]$ is distinct from $Y$.
Since $N_{D^*}^-[u]$ contains $u$, $N_{D^*}^-[u]=X\setminus\{v\}$ by ($\dag$).
Since $|X| \le i+1$, $|X\setminus\{v\}| \le i$ and so $d_{D^*}^-(u) \le i-1$.
Adding the vertex $v$ and the arc $(v, u)$ to $D^*$ results in an $(i, 1)$ digraph whose phylogeny graph is $G$.
\end{proof}

The {\it union} of two graphs $G$ and $H$ is the graph having its vertex set $V(G) \cup V(H)$ and edge set $E(G) \cup E(H)$.
If $V(G) \cap V(H) = \emptyset$, we refer to their union as a \emph{disjoint} union.

\begin{Prop}\label{prop:(1,1)}
For a graph $G$, the following statements are equivalent.
\begin{itemize}
  \item[(i)] $G$ is a $(1, j)$ phylogeny graph and an $(i, 1)$ phylogeny graph for some positive integers $i$ and $j$;
  \item[(ii)] $G$ is a disjoint union of paths;
  \item[(iii)] $G$ is a $(1, 1)$ phylogeny graph.
\end{itemize}
\end{Prop}
\begin{proof}
By Theorems~\ref{thm:(1,j) realizable} and~\ref{thm:(i,1) realizable}, it is immediately true that (ii) is equivalent (iii).
Obviously, (iii) implies (i).
Now we show that (i) implies (ii).
By Theorem~\ref{thm:(1,j) realizable}, $G$ is a forest.
If $G$ has a vertex of degree at least three, then $K(G)$ contains a triangle as each edge in $G$ is a maximal clique, which contradicts Theorem~\ref{thm:(i,1) realizable}.
Therefore each vertex in $G$ has degree at most two and so $G$ is a disjoint union of paths.
\end{proof}

\begin{Rem}
Theorems~\ref{thm:(1,j) realizable} and ~\ref{thm:(i,1) realizable} tell us that an $(i, j)$ phylogeny graph for positive integers $i$ and $j$ with $i=1$ or $j=1$ is diamond-free and chordal.
\end{Rem}

%

\section{$(2, j)$ phylogeny graphs}
In this section, we focus on phylogeny graphs of $(2, j)$ digraphs for a positive integer $j$. We thought that it is worth studying them in the context that a child has two biological parents in most species.

For an acyclic digraph $D$, an edge is called a \emph{cared edge} in $P(D)$ if the edge belongs to the competition graph $C(D)$ but not to the $U(D)$.
For a cared edge $xy \in P(D)$, there is a common out-neighbor $v$ of $x$ and $y$ and it is said that $xy$ \emph{is taken care} of by $v$ or that $v$ \emph{takes care} of $xy$.
A vertex in $D$ is called a \emph{caring vertex} if an edge of $P(D)$ is taken care of by the vertex.

For example, the edges $v_2v_3$, $v_2v_6$, $v_2v_7$, $v_4v_5$, and $v_5v_6$ of $P(D)$ in Figure~\ref{fig:counterexample} are cared edges
and the vertices $v_1$, $v_4$, $v_4$, $v_3$, and $v_7$ are vertices taking care of $v_2v_3$, $v_2v_6$, $v_2v_7$, $v_4v_5$, and $v_5v_6$, respectively.

\begin{Prop}\label{prop:not on H}
Suppose that the phylogeny graph of a $(2,j)$ digraph $D$ contains a hole $H$ for a positive integer $j$.
Then no vertex on $H$ takes care of an edge on $H$.
\end{Prop}
\begin{proof}
Suppose, to the contrary, that there exists a vertex $v$ on $H$ which takes care of an edge $xy$ on $H$.
Then $\{x, y, v\}$ forms a triangle in $P(D)$, so $yv$ or $vx$ is a chord of $H$ in $P(D)$ and we reach a contradiction.
\end{proof}

Given a $(2,j)$ digraph $D$, suppose that $P(D)$ has a hole $H$ and $e_1, e_2, \ldots, e_t$ are the cared edges on $H$.
Let $w_1, w_2, \ldots, w_t$ be vertices taking care of $e_1, e_2, \ldots, e_t$, respectively.
Since the indegree of $w_i$ is at most two in $D$ for $i=1, \ldots, t$, $w_1, w_2, \ldots, w_t$ are distinct.
We let $W=\{w_1, w_2, \ldots, w_t\}$ and call $W$ a \emph{set extending} $H$ by extending the notion introduced in Lee~{\it et al.}~\cite{lee2017phylogeny}.
By Proposition~\ref{prop:not on H},
\begin{equation}\label{eqn:caringvertex}
W \subset V(D) \setminus V(H).
\end{equation}
Therefore we may obtain a cycle in $U(D)$ from $H$ by replacing each edge $e_i$ with a path of length two from one end of $e_i$ to the other end of $e_i$ with the interior vertex $w_i$.
We call such a cycle the \emph{cycle obtained from $H$ by $W$}.
Let $L$ be the subgraph of $U(D)$ induced by $V(H) \cup W$.
We call $L$ the \emph{subgraph of $U(D)$ obtained from $H$ by $W$}.
These notions extend the ones introduced in Lee~{\it et al.}~\cite{lee2017phylogeny}.


Lee~{\it et al.}~\cite{lee2017phylogeny} showed that, for a $(2,2)$ digraph $D$ such that the holes of $P(D)$ are mutually vertex-disjoint and no hole in $U(D)$ has length $4$ or $6$,  the number of holes in $U(D)$ is greater than or equal to the number of holes in $P(D)$.

\begin{Thm}[\cite{lee2017phylogeny}]~\label{Thm:13}
Let $H$ be a hole of the phylogeny graph $P(D)$ of a $(2,2)$ digraph $D$.
Then there is a hole $\phi(H)$ in the underlying graph $U(D)$ of $D$ such that
\begin{itemize}
\item
$\phi(H)$ equals $H$ if $H$ is a hole in $U(D)$;
\item
$\phi(H)$ is a hole in $U(D)$ only containing vertices in the subgraph obtained from $H$ by a set extending $H$ otherwise.
\end{itemize}
Moreover, if the holes of $P(D)$ are mutually vertex-disjoint and no hole in $U(D)$ has length $4$ or $6$,
then there exists an injective map from the set of holes in $P(D)$ to the set of holes in $U(D)$.
\end{Thm}

We shall devote the first part of this section to extending the above theorem given in \cite{lee2017phylogeny}.
To do so, we need the following lemmas.

\begin{Lem}[\cite{choi2018new}]~\label{lem:chord}
Given a graph $G$ and a cycle $C$ of $G$ with length at least four,
suppose that a section $Q$ of $C$ forms an induced path of $G$ and contains a path $P$  with length at least two none of whose internal vertices is incident to a chord of $C$ in $G$.
Then $P$ can be extended to a hole $H$ in $G$ so that $V(P) \subsetneq V(H) \subset V(C)$ and $H$ contains a vertex on $C$ not on $Q$.
\end{Lem}
%
%

\begin{Lem}\label{lem:least f-value in a hole}
Let $D$ be a $(2, j)$ digraph and $f$ be an acyclic labeling of $D$ for a positive integer $j$.
In addition, let $H$ be a hole of $P(D)$, $W$ be a set extending $H$, and $w$ be a vertex  with the least $f$-value in $V(H) \cup W$. Then $w \in W$.
Moreover, there is a hole $\phi(H)$ in $U(D)$ such that $w \in V(\phi(H))$ and $V(\phi(H)) \subset V(H) \cup W$.
\end{Lem}
\begin{proof}
Let $H = u_1u_2 \cdots u_lu_1$ for an integer $l \ge 4$.
To reach a contradiction, we suppose that $w  \in V(H)$.
Without loss of generality, we may assume that $w=u_1$.
Suppose that $u_1u_2$ and $u_1u_l$ are edges of $U(D)$.
Then, since $u_1$ has the least $f$-value in $V(H)$, $(u_2, u_1) \in A(D)$ and $(u_l, u_1) \in A(D)$ and so $\{u_1, u_2, u_l\}$ forms a triangle in $P(D)$, which is a contradiction to the supposition that $H$ is a hole in $P(D)$.
Therefore $u_1u_2$ or $u_1u_l$ is a cared edge in $P(D)$.
Without loss of generality, we may assume $u_1u_2$ is a cared edge in $P(D)$.
Then $u_1$ and $u_2$ have a common out-neighbor, say $v$, in $W$, which implies that $f(v) < f(u_1)$. Thus we have reached a contradiction and so $w \in W$.

Now we  show that the ``moreover" part of the lemma statement is true.
Let $C$ be  the cycle in $U(D)$ obtained from $H$ by $W$.
Without loss of generality, we may assume that $u_1u_l$ is taken care of by $w$.
Then $(u_1, w) \in A(D)$ and $(u_l, w) \in A(D)$.

Suppose, to the contrary, that $C$ has a chord which is incident to $w$ in $U(D)$.
Let $xw$ be a chord of $C$ in $U(D)$.
Then $x \notin \{u_1, u_l\}$.
Moreover, since $w$ has the least $f$-value in $V(H) \cup W$, $(w, x) \notin A(D)$.
Then $u_1$, $u_l$, and $x$ are in-neighbors of $w$ in $D$, which contradicts the hypothesis that $D$ is a $(2, j)$ digraph.
Hence there is no chord of $C$ which is incident to $w$ in $U(D)$.
Since $u_1u_l$ is a cared edge in $P(D)$, $u_1wu_l$ is an induced path in $U(D)$.
By applying Lemma~\ref{lem:chord} for $P=Q=u_1wu_l$, we may conclude that ``moreover" part of the lemma statement is true.
\end{proof}

Now we are ready to extend Theorem~\ref{Thm:13} to not only make it valid for $(2,j)$ digraphs but also strengthen it.

\begin{Thm} \label{thm:new, hole of the phylogeny graph of (2,j) digraph}
For a positive integer $j$, let $H$ be a hole of the phylogeny graph $P(D)$ of a $(2,j)$ digraph $D$.
Then there is a hole in $U(D)$ which only contains vertices in the subgraph of $U(D)$ obtained from $H$ by a set extending $H$.
Moreover, if $P(D)$ has a hole and the holes of $P(D)$ are mutually edge-disjoint, then there exists an injective map from the set of holes in $P(D)$ to the set of holes in $U(D)$.
\end{Thm}
\begin{proof}
The first part of this theorem is immediately true by Lemma~\ref{lem:least f-value in a hole}.

To show the second part of the theorem statement, we assume that $P(D)$ has a hole and the holes in $P(D)$ are mutually edge-disjoint.
Let $f$ be an acyclic labeling of $D$, $\{H_1, \ldots, H_l\}$ be the set of holes in $P(D)$, and $W_i$ be a set extending $H_i$ for each $i=1, \ldots, l$. Let $w_i$ be the vertex with the least $f$-value in $V(H_i) \cup W_i$ for each $i=1, \ldots, l$. Then, by Lemma~\ref{lem:least f-value in a hole}, $w_i \in W_i$ and there exists a hole $\phi(H_i)$ such that $w_i \in V(\phi(H_i))$ and $V(\phi(H_i)) \subset V(H_i) \cup W_i$
for each $i=1, \ldots, l$.
At this point, we may regard $\phi$ as a map from the set of the holes in $P(D)$ to the set of holes in $U(D)$.

In the following, we show that $\phi$ is injective.
Suppose, to the contrary, that $\phi(H_j)=\phi(H_k)$ for some $j$ and $k$ satisfying $1 \le j < k \le l$.
Since $w_i$ is the vertex with the least $f$-value in $V(H_i) \cup W_i$ and $V(\phi(H_i) )\subset V(H_i) \cup W_i$, $w_i$ has the least $f$-value in $V(\phi(H_i))$ for each $i\in \{j,k\}$.
Then, since $\phi(H_j)=\phi(H_k)$, $w_j =w_k$ and so $w_j \in W_j \cap W_k$.
Thus $w_j$ has two in-neighbors on $H_j$ and two in-neighbors on $H_k$ in $D$. Then, by the hypothesis that $H_j$ and $H_k$ are edge-disjoint, $w_j$ has at least three distinct in-neighbors in $D$, which violates the indegree restriction on $D$.
Hence $\phi(H_j) \neq \phi(H_k)$ for any $j$ and $k$ satisfying $1 \le j < k \le l$ and we have shown that $\phi$ is injective.
\end{proof}

The underlying graph of an $(i,j)$ digraph $D$ being chordal does not guarantee that the phylogeny graph of $D$ is chordal. For example, the underlying graph of the $(3, 2)$ digraph given in Figure~\ref{fig:counterexample} is chordal whereas its phylogeny graph has a hole $v_2v_3v_5v_6v_2$. However, if $i \le 2$ or $j=1$, then it does guarantee by the above theorem together with Theorems~\ref{thm:(1,j) realizable} and~\ref{thm:(i,1) realizable}. As a matter of fact, we have shown the following theorem.
\begin{Thm}\label{thm:(2,j) chordal}
Let $\DDD^*_{i,j}$ be the set of $(i, j)$ digraphs whose underlying graphs are chordal for positive integers $i$ and $j$.
Then the phylogeny graph of $D$ is chordal for any $D \in \DDD^*_{i,j}$ if and only if $i \le 2$ or $j=1$.
\end{Thm}
By Theorem~\ref{thm:(2,j) chordal}, the phylogeny graph of a $(2,j)$ digraph $D$ is chordal if the underlying graph of $D$ is chordal for any positive integer $j$. By the way, if $j=2$, then the underlying graph being chordal guarantees not only $P(D)$ being chordal but also $P(D)$ being planar, which will be to be shown later in this section. By the way, Lee~{\it et al.}~\cite{lee2017phylogeny} showed that a $(2, 2)$ phylogeny graph is $K_5$-free.
\begin{Thm}[\cite{lee2017phylogeny}]~\label{thm:K5-free}
For any $(2, 2)$ digraph $D$, the phylogeny graph of $D$ is $K_5$-free.
\end{Thm}
We shall extend this theorem in two aspects. On one hand, we find a sharp upper bound for the clique number of $(2, j)$ phylogeny graph for any positive integer $j$. On the other hand, we show that the phylogeny graph $P(D)$ of a $(2,2)$ digraph $D$ is planar if the underlying graph of $D$ is chordal by showing that $P(D)$ is $K_5$-minor-free and $K_{3,3}$-minor-free.

For a positive integer $k$, a graph $G$ is \emph{$k$-degenerate} if any subgraph of $G$ contains a vertex having at most $k$ neighbors in it.

\begin{Lem}\label{lem:degenerate}
For a positive integer $j$, every $(2, j)$ phylogeny graph is $(j+2)$-degenerate.
\end{Lem}
\begin{proof}
Let $D$ be a $(2, j)$ digraph for a positive integer $j$ and $f$ be an acyclic labeling of $D$.
We take a subgraph $H$ of $P(D)$ and the vertex $u$ which has the least $f$-value in $V(H)$.
Then the out-neighbors of $u$ in $D$ cannot be in $V(H)$.
Thus an edge incident to $u$ in $H$ is either a cared edge or the edge in $U(D)$ corresponding to an arc incoming toward $u$ in $D$.
Since $u$ has at most $j$ out-neighbors and each of the out-neighbors has at most one in-neighbor other than $u$ in $D$, there are at most $j$ cared edges which are incident to $u$ in $H$.
Moreover, since $u$ has at most two in-neighbors in $D$, there are at most two edges incident to $u$ in $H$ which correspond to arcs incoming toward $u$ in $D$.
Thus $u$ has degree at most $j+2$ in $H$.
Since $H$ was arbitrarily chosen, $P(D)$ is $(j+2)$-degenerate.
\end{proof}
The following theorem gives a sharp upper bound for the clique number of $(2, j)$ phylogeny graph for any positive integer $j$ to extend the Theorem~\ref{thm:K5-free}.
\begin{Thm}\label{thm:cliquenumber}
Let $D$ be a $(2, j)$ digraph for a positive integer $j$. Then
\[
\omega(P(D)) \le
\begin{cases}
  j+2 & \mbox{if } j \le 2; \\
  j+3 & \mbox{otherwise;}
\end{cases}
\]
and the inequalities are tight.
\end{Thm}
\begin{proof}
It is known that if a graph $G$ is $k$-degenerate, then $\omega(G) \le k+1$.
Thus, by Lemma~\ref{lem:degenerate}, $\omega(P(D)) \le j+3$.
By Theorems~\ref{thm:(i,1) realizable} and \ref{thm:K5-free}, $\omega(P(D)) \le j+2$ if $j \le 2$.

The inequality is tight for $j \le 2$ by the digraphs given in Figure~\ref{fig:construct}.
\begin{figure}
\begin{center}
\begin{tikzpicture}[x=1.0cm, y=1.0cm]

    \vertex (b2) at (0,3) [label=left:$v_2$]{};
    \vertex (b3) at (2,2) [label=right:$v_3$]{};

    \vertex (b4) at (0,1) [label=left:$v_4$]{};


    \path
 (b2) edge [->,thick] (b3)
 (b4) edge [->,thick] (b3)

;
\end{tikzpicture}
\qquad \qquad
\begin{tikzpicture}[x=1.0cm, y=1.0cm]

    \vertex (b2) at (0,3) [label=left:$v_2$]{};
    \vertex (b3) at (2,3) [label=right:$v_3$]{};

    \vertex (b4) at (0,1) [label=left:$v_4$]{};
    \vertex (b5) at (2,1) [label=right:$v_5$]{};
    \vertex (b6) at (3,2) [label=right:$v_6$]{};

    \path
 (b2) edge [->,thick] (b3)
 (b2) edge [->,thick] (b5)
 (b4) edge [->,thick] (b3)
 (b4) edge [->,thick] (b5)
 (b3) edge [->,thick] (b6)
 (b5) edge [->,thick] (b6)
 ;

\end{tikzpicture}

\end{center}
\caption{A $(2, 1)$ digraph and $(2, 2)$ digraph whose phylogeny graphs contain $K_3$ and $K_4$, respectively.}
\label{fig:construct}
\end{figure}
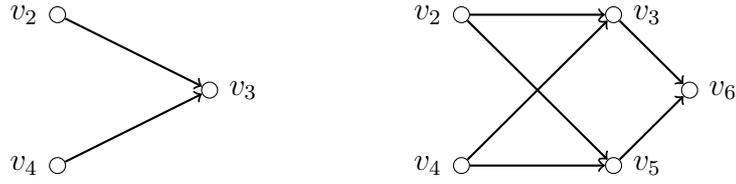
To show that the inequality is tight for $j \ge 3$, we construct a $(2, j)$ digraph in the following way.
We start with an empty digraph $D_0$ with vertex set $\{v_1, \ldots v_{j+3}\}$.
We add to $D_0$ the vertices $a_{1,2}, \ldots, a_{1, j+1}$ and the arcs $(v_1, a_{1,i})$, $(v_i,a_{1,i})$ for $i=2$, $\ldots$, $j+1$ and arcs $(v_{j+2},v_1)$, $(v_{j+3}, v_1)$ to obtain a digraph $D_1$. Then $D_1$ is a $(2,j)$-digraph with every vertex except $v_1$ having outdegree at most one and
\[
E_1:=\{v_{j+2}v_{j+3}\} \cup \{v_1v_i \mid i=2, \ldots, j+3\}
 \]
is an edge set of $P(D_1)$.
We add to $D_1$ the vertices $a_{2,3}, \ldots, a_{2, j-1}, a_{2, j+1}, a_{2, j+2}$ and the arcs $(v_2, a_{2,i})$, $(v_i,a_{2,i})$ for each $i \in [j+2] \setminus \{1,2,j\}$ and arcs $(v_{j},v_2)$, $(v_{j+3}, v_2)$ to obtain a digraph $D_2$. Then $D_2$ is a $(2,j)$-digraph with every vertex except $v_1$ and $v_2$ having outdegree at most two and
\[E_2:=E_1 \cup \{v_jv_{j+3}\} \cup \{v_2v_i \mid i=3, \ldots, j+3\}\]
is an edge set of $P(D_2)$.

For each $\ell \in [j-1] \setminus \{1,2\}$, we add to $D_{\ell-1}$ the vertices $a_{\ell,\ell+1}, \ldots, a_{\ell, j+1}$ and the arcs $(v_\ell, a_{\ell,i})$, $(v_i,a_{\ell,i})$ for $i=\ell+1$, $\ldots$, $j+1$ and arcs $(v_{j+2},v_\ell)$, $(v_{j+3}, v_\ell)$ to obtain a digraph $D_\ell$.
Then, for each $\ell \in [j-1] \setminus \{1,2\}$ is a $(2,j)$-digraph with every vertex except $v_1$, $\ldots$, and $v_l$ having outdegree at most $\ell$ and
\[E_\ell:=E_{\ell-1}\cup \{v_\ell v_i \mid i=\ell+1, \ldots, j+3\}\]
is an edge set of $P(D_\ell)$.
Therefore $v_i$ is adjacent to each of $v_1$, $\ldots$, $v_{j+3}$ except itself for $i=1$, $\ldots$, $j-1$.
Now we add to $D_{j-1}$ the arcs $(v_{j+3}, v_{j+1})$, $(v_{j+2}, v_{j})$, and $(v_{j+1}, v_{j})$ to obtain a $(2, j)$ digraph $D_{j}$.
Clearly, $v_j$, $\ldots$, $v_{j+3}$ are mutually adjacent in $P(D_j)$ (recall that the edges $v_{j+2}v_{j+3}$ and $v_jv_{j+3}$ are contained in $E_1$ and $E_2$, respectively).
Thus $v_1$, $\ldots$, $v_{j+3}$ form a clique of size $j+3$ in $P(D_j)$.
\end{proof}
From Theorems~\ref{thm:(1,j) realizable} and~\ref{thm:(i,1) realizable}, we know that the clique number of a $(1, j)$ phylogeny graph is at most two and the clique number of an $(i, 1)$ phylogeny graph is at most $i+1$ for any positive integers $i$ and $j$.

In the rest of this section, we shall show that the phylogeny graph $P(D)$ of a $(2,2)$ digraph $D$ is planar if the underlying graph of $D$ is chordal.

The following lemma is a known fact.

\begin{Lem}\label{lem:contraction}
The class of chordal graphs is closed under contraction.
\end{Lem}

We denote by $G \cdot e$ the graph obtained by contracting a graph $G$ by an edge $e$ in $G$.

\begin{Lem}\label{lem:contraction hole}
For a graph $G$ and two adjacent vertices $u$ and $v$ in $G$, let $K$ be a clique with at least three vertices in $G \cdot uv$.
If $z$ is the vertex in $K$ obtained by identifying $u$ and $v$, then one of the following is true:
\begin{itemize}
  \item $K \setminus \{z\} \subset N_G(u)$;
  \item $K \setminus \{z\} \subset N_G(v)$;
  \item the subgraph of $G$ induced by $(K \setminus \{z\}) \cup \{u, v\}$ contains a hole in $G$.
\end{itemize}
\end{Lem}
\begin{proof}
Suppose that $K \setminus \{z\} \not\subset N_G(u)$ and  $K \setminus \{z\} \not\subset N_G(v)$.
Then there is a vertex $w$ and $x$ in $K \setminus \{z\}$ such that $w$ is not adjacent to $u$ and $x$ is not adjacent to $v$.
Since $K$ is a clique in $G \cdot uv$, $w$ and $x$ are adjacent to $v$ and $u$, respectively, in $G$, and so $uxwvu$ is a hole in $G$.
\end{proof}

\begin{Lem}\label{lem:chordal K-minor-free}
A chordal graph $G$ is $K_{\omega(G)+1}$-minor-free.
\end{Lem}
\begin{proof}
Denote $\omega(G)$ by $\omega$ for simplicity's sake.
Suppose, to the contrary, that $G$ contains $K_{\omega+1}$ as a minor.
Then, since $K_{\omega+1}$ is complete, $G$ contains an induced subgraph $H$ such that $K_{\omega+1}$ is obtained from $H$ by only contraction.
Moreover, we may regard $H$ as an induced subgraph of $G$ for which the smallest number of contractions are required to obtain $K_{\omega+1}$.
Then, since $G$ is chordal, $H$ is also chordal.
Clearly $H$ is $K_{\omega+1}$-free, so at least one edge of $H$ is contracted to obtain $K_{\omega+1}$.
Let $uv$ be the last edge contracted to obtain $K_{\omega+1}$ from $H$.
Let $L$ be the second last graph obtained in the series of contractions to obtain $K_{\omega+1}$ from $H$, that is, $L \cdot uv=K_{\omega+1}$.
Then, by Lemma~\ref{lem:contraction hole}, $V(L) \setminus \{u, v\} \subset N_L(u)$ or $V(L) \setminus \{u, v\} \subset N_L(v)$ or $L$ contains a hole.
If $V(L) \setminus \{u, v\} \subset N_L(u)$ or $V(L) \setminus \{u, v\} \subset N_L(v)$, then $L - v$ or $L-u$ is isomorphic to $K_{\omega+1}$, which contradicts the choice of $H$.
Thus $V(L) \setminus \{u, v\} \not\subset N_L(u)$ and $V(L) \setminus \{u, v\} \not\subset N_L(v)$, and so $L$ contains a hole.
However, since $H$ is chordal, by Lemma~\ref{lem:contraction}, $L$ is chordal and we reach a contradiction.
\end{proof}

\begin{Thm}\label{thm:K_5-minor-free}
For a positive integer $j$ and a $(2,j)$ digraph, if its underlying graph is chordal, then its phylogeny graph is $K_{j+3}$-minor-free if $j \le 2$ and $K_{j+4}$-minor-free if $j \ge 3$.
\end{Thm}
\begin{proof}
Let $D$ be a $(2, j)$ digraph for a positive integer $j$ whose underlying graph is chordal.
Then, by Theorem~\ref{thm:(2,j) chordal}, $P(D)$ is chordal.
Moreover,
\[
\omega(P(D)) \le
\begin{cases}
  j+2, & \mbox{if } j \le 2 \\
  j+3, & \mbox{otherwise}.
\end{cases}
\]
by Theorem~\ref{thm:cliquenumber}.
Thus $P(D)$ is $K_{j+3}$-free (resp.\ $K_{j+4}$-free) if $j \le 2$ (resp.\ $j \ge 3$).
By Lemma~\ref{lem:chordal K-minor-free}, $P(D)$ is $K_{j+3}$-minor-free (resp.\ $K_{j+4}$-minor-free) if $j \le 2$ (resp.\ $j \ge 3$).
\end{proof}

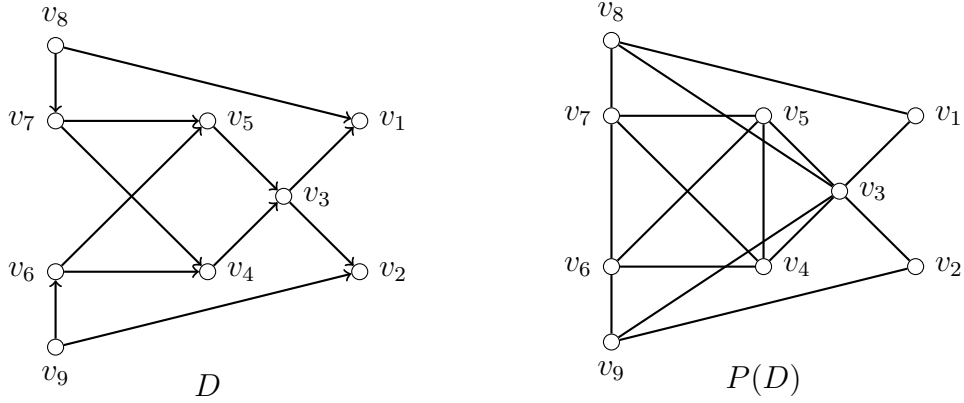
\begin{figure}
\begin{center}
\begin{tikzpicture}[x=1.0cm, y=1.0cm]

    \vertex (b2) at (0,3) [label=left:$v_7$]{};
    \vertex (b3) at (2,3) [label=right:$v_5$]{};

    \vertex (b4) at (0,1) [label=left:$v_6$]{};
    \vertex (b5) at (2,1) [label=right:$v_4$]{};
    \vertex (b6) at (3,2) [label=right:$v_3$]{};
    \vertex (b7) at (0,4) [label=above:$v_8$]{};
    \vertex (b8) at (0,0) [label=below:$v_9$]{};
    \vertex (b9) at (4,3) [label=right:$v_1$]{};
    \vertex (b10) at (4,1) [label=right:$v_2$]{};

    \path
 (b2) edge [->,thick] (b3)
 (b2) edge [->,thick] (b5)
 (b4) edge [->,thick] (b3)
 (b4) edge [->,thick] (b5)
 (b3) edge [->,thick] (b6)
 (b5) edge [->,thick] (b6)
 (b7) edge [->,thick] (b2)
 (b8) edge [->,thick] (b4)
 (b7) edge [->,thick] (b9)
 (b8) edge [->,thick] (b10)
 (b6) edge [->,thick] (b9)
 (b6) edge [->,thick] (b10)
 ;
 \draw (2, -0.5) node{$D$};
\end{tikzpicture}
\qquad \qquad
\begin{tikzpicture}[x=1.0cm, y=1.0cm]

    \vertex (b2) at (0,3) [label=left:$v_7$]{};
    \vertex (b3) at (2,3) [label=right:$v_5$]{};

    \vertex (b4) at (0,1) [label=left:$v_6$]{};
    \vertex (b5) at (2,1) [label=right:$v_4$]{};
    \vertex (b6) at (3,2) [label=right:$v_3$]{};
    \vertex (b7) at (0,4) [label=above:$v_8$]{};
    \vertex (b8) at (0,0) [label=below:$v_9$]{};
    \vertex (b9) at (4,3) [label=right:$v_1$]{};
    \vertex (b10) at (4,1) [label=right:$v_2$]{};

    \path
 (b2) edge [-,thick] (b3)
 (b2) edge [-,thick] (b5)
 (b4) edge [-,thick] (b3)
 (b4) edge [-,thick] (b5)
 (b3) edge [-,thick] (b6)
 (b5) edge [-,thick] (b6)
 (b7) edge [-,thick] (b2)
 (b8) edge [-,thick] (b4)
 (b7) edge [-,thick] (b9)
 (b8) edge [-,thick] (b10)
 (b6) edge [-,thick] (b9)
 (b6) edge [-,thick] (b10)

 (b2) edge [-,thick] (b4)
 (b3) edge [-,thick] (b5)
 (b7) edge [-,thick] (b6)
 (b8) edge [-,thick] (b6)
 ;
\draw (2,-0.5) node{$P(D)$};
\end{tikzpicture}

\end{center}
\caption{A $(2, 2)$ digraph $D$ whose phylogeny graph contains $K_5$ as a minor.}
\label{fig:notplanar}
\end{figure}
The above theorem is false for a $(2, j)$ digraph whose underlying graph is non-chordal (see Figure~\ref{fig:notplanar}).

\begin{Cor}\label{cor:K_5-minor-free}
If the underlying graph of a $(2,2)$ digraph is chordal, then its phylogeny graph is $K_{5}$-minor-free.
\end{Cor}

In the following, we show that the phylogeny graph of $(2,2)$ digraph whose underlying graph is chordal is $K_{3,3}$-minor-free.

The \emph{join} of two graphs $G_1$ and $G_2$ is denoted by $G_1 \vee G_2$ and has the vertex set $V(G_1) \cup V(G_2)$ and the edge set $E(G_1)\cup E(G_2) \cup \{xy \mid x \in G_1 \text{ and } y \in G_2 \}$.
Let $I_n$ denote a set of $n$ isolated vertices in a graph for a positive integer $n$.

\begin{Lem}\label{lem:(K_3,I_3)-minor-free}
For any $(2,2)$ digraphs, if its underlying graph is chordal, then its phylogeny graph is $K_3 \vee I_3$-minor-free.
\end{Lem}
\begin{proof}
Let $G$ be the phylogeny graph of a $(2, 2)$ digraph $D$ whose underlying graph is chordal.
Then, by Theorem~\ref{thm:(2,j) chordal} and Corollary~\ref{cor:K_5-minor-free}, $G$ is chordal and $K_5$-minor-free.

Suppose, to the contrary, that $K_3 \vee I_3$ is a minor of $G$.
Then $G$ contains a subgraph $H$ such that either $H=K_3 \vee I_3$ or $K_3 \vee I_3$ is obtained from $H$ by using edge deletions or contractions.
Let $f$ be an acyclic labeling of $D$.

Suppose that $H=K_3 \vee I_3$.
If $H$ is not an induced subgraph of $G$, then two vertices of $I_3$ are adjacent in $G$, and so $K_5$ is a subgraph of $G$, which is impossible.
Thus $H$ is an induced subgraph of $G$.
We denote the vertices of $K_3$ in $H$ by $x_1, x_2, x_3$ and the vertices of $I_3$ in $H$ by  $y_1, y_2, y_3$. We may assume that $f(x_1)<f(x_2)<f(x_3)$ and $f(y_1)<f(y_2)<f(y_3)$.

If $f(x_1) < f(y_1)$, then the outdegree of $x_1$ in the subdigraph $D_H$ of $D$ induced by $V(H)$ is zero, which implies $d_H(x_1) \le 4$ (recall that $D$ is a $(2,2)$ digraph), a contradiction.
Thus
\begin{equation}\label{eq:inequalities}
f(y_1)<f(x_1)<f(x_2)<f(x_3) \quad \text{and} \quad f(y_1)<f(y_2)<f(y_3).
\end{equation}
If $x_3$ has two in-neighbors in $D_H$, then they must be $y_2$ and $y_3$, which implies their being adjacent in $G$, a contradiction.
Therefore $x_3$ has at most one in-neighbor in $D_H$.
Since $D$ is a $(2,2)$ digraph and $d_H(x_3)=5$, $x_3$ has exactly one in-neighbor and two out-neighbors in $D_H$, and two cared edges in $H$ are incident to $x_3$.
The in-neighbor of $x_3$ in $D_H$ is $y_2$ or $y_3$ by \eqref{eq:inequalities}.

Let $y$ be the in-neighbor of $x_3$ in $D_H$.
Then $y \in \{y_2, y_3\}$ and $f(y) > f(x_3)$. Thus, by \eqref{eq:inequalities}, none of  $x_1$, $x_2$, and $y_1$ is an in-neighbor of $y$ in $D_H$.
Since $D$ is a $(2, 2)$ digraph, $y$ has at most one out-neighbor other than $x_3$ in $D_H$. Then, since $d_H(y)=3$, by \eqref{eq:inequalities}, one of $x_2y$ and $x_1y$ is a cared edge in $G$ taken care of by $x_1$ or $x_2$.
Since $f(x_1) < f(x_2)$, $x_2y$ is a cared edge in $G$ taken care of by $x_1$.

Let $v$ be a vertex joined to $x_3$ by a cared edge in $H$.
Then $x_3$ and $v$ have a common out-neighbor in $D$.
Since $x_3$ has all of its two out-neighbors in $D_H$, the common out-neighbors of $x_3$ and $v$ should be in $H$.
Since there are two cared edges incident to $x_3$ in $H$, the two out-neighbors of $x_3$ take care of those two cared edges incident to $x_3$.
Since $y_1$ has the least $f$-value among the vertices in $H$, $y_1$ cannot be none of the other ends of two cared edges incident to $x_3$ in $G$.
Hence $y_1$ must be one of the two out-neighbors of $x_3$ in $D_H$  which takes care of a cared edge incident to $x_3$.
Since $\{y_1, y_2, y_3\}$ is an independent set in $G$, neither $y_2$ nor $y_3$ can be an in-neighbor of $y_1$ in $D_H$.
Thus $x_1$ or $x_2$ is the vertex joined to $x_3$ which is taken care of by $y_1$ in $D_H$.

If $x_1$ is an in-neighbor of $y_1$ in $D_H$, then $x_1y_1x_3yx_1$ is a hole in $U(D)$ since $\{y, y_1\} \subset I_3$ and $x_1x_3$ is a cared edge in $G$ which is not an edge in $U(D)$.
Thus $x_2$ is an in-neighbor of $y_1$ in $D_H$.
In the following, we shall claim that $x_1x_2y_1x_3yx_1$ is a hole in $U(D)$ to reach a contradiction.
Since $\{y, y_1\} \subset I_3$, $y$ and $y_1$ are not adjacent in $U(D)$.
Since $x_2x_3$ is a cared edge in $G$, $x_2$ and $x_3$ are not adjacent in $U(D)$.
If $x_1x_3$ is an edge of $U(D)$, then there is an arc $(x_3, x_1)$ since $f(x_1) < f(x_3)$, which contradicts the indegree condition on $x_1$.
Therefore $x_1$ and $x_3$ are not adjacent in $U(D)$.
By applying a similar argument, we may show that neither $x_1$ and $y_1$ nor $y$ and $x_2$ are adjacent in $U(D)$.

Thus $H \neq K_3 \vee I_3$ and so $K_3 \vee I_3$ is obtained from $H$ by using edge deletions or contractions.
Then, $K_3 \vee I_3$ may be obtained from the subgraph of $G$ induced by $V(H)$ by using edge deletions or contractions, so we may assume that $H$ as an induced subgraph of $G$.
Then $H$ is chordal.
If an edge deletion was required to obtain $K_3 \vee I_3$ from $H$, then it would mean that $G$ contains $K_5$ as a minor, which is impossible.
Thus, we may assume that $K_3 \vee I_3$ is obtained from $H$ by only contractions.

Let $H^*$ be a graph obtained from $H$ by applying the smallest number of contractions to contain $K_3 \vee I_3$ as a subgraph.
Since $H$ is chordal, $H^*$ is chordal by Lemma~\ref{lem:contraction}.

Let $x_1, x_2, x_3$ be the vertices of $K_3$ and $y_1, y_2, y_3$ be the vertices of $I_3$ for $K_3 \vee I_3$ contained in $H^*$.
Let $H'$ be the graph to which the last contraction is applied in the process of obtaining $H^*$ and $e=uv$ be the edge contracted lastly.
Then $H'$ is chordal by Lemma~\ref{lem:contraction}.
By the choice of $H^*$, $u$ and $v$ are identified to become a vertex in $\{x_1, x_2, x_3, y_1, y_2, y_3\}$.

{\it Case 1}. The vertices $u$ and $v$ are identified to become one of $y_1$, $y_2$, $y_3$.
Without loss of generality, we may assume that $u$ and $v$ are identified to become the vertex $y_3$.
By Lemma~\ref{lem:contraction hole}, $\{x_1, x_2, x_3\} \subset N_{H'}(u)$ or $\{x_1, x_2, x_3\} \subset N_{H'}(v)$ or $\{x_1, x_2, x_3, u, v\}$ contains a hole in $H'$.
Since $H'$ is chordal, $\{x_1, x_2, x_3\} \subset N_{H'}(u)$ or $\{x_1, x_2, x_3\} \subset N_{H'}(v)$.
Then $\{x_1, x_2, x_3, y_1, y_2, u\}$ or  $\{x_1, x_2, x_3, y_1, y_2, v\}$ forms $K_3 \vee I_3$ in $H'$, which contradicts the choice of $H^*$.

{\it Case 2}. The vertices $u$ and $v$ are identified to become one of $x_1$, $x_2$, $x_3$. Then each of $y_1$, $y_2$, $y_3$ is adjacent to one of $u$, $v$ in $H'$.
Without loss of generality, we may assume that $u$ and $v$ are identified to become the vertex $x_3$.
By Lemma~\ref{lem:contraction hole}, $\{x_1, x_2\} \subset N_{H'}(u)$ or $\{x_1, x_2\} \subset N_{H'}(v)$ or $\{x_1, x_2, u, v\}$ contains a hole in $H'$.
Since $H'$ is chordal, $\{x_1, x_2\} \subset N_{H'}(u)$ or $\{x_1, x_2\} \subset N_{H'}(v)$.
Without loss of generality, we may assume that $\{x_1, x_2\} \subset N_{H'}(u)$.
If $u$ is adjacent to each of $y_1$, $y_2$, $y_3$, then $\{x_1, x_2, u, y_1, y_2, y_3\}$ forms $K_3 \vee I_3$ in $H'$, a contradiction to the choice of $H^*$.
Thus $u$ is not adjacent to one of $y_1$, $y_2$, $y_3$ in $H'$.
Without loss of generality, we may assume that $u$ is not adjacent to $y_3$ in $H'$.
Then $v$ is adjacent to $y_3$ in $H'$.
If $v$ is not adjacent to one of $x_1$ and $x_2$, then  $x_1y_3vux_1$ or $x_2y_3vux_2$ is a hole in $H'$ and we reach a contradiction.
Thus $v$ is adjacent to $x_1$ and $x_2$.
If one of $y_1$, $y_2$ is adjacent to both of $u$ and $v$, then $x_1$, $x_2$, $ u$, and $v$ together with it form $K_5$ in $H'$, a contradiction.
Therefore $\{ N_{H'}(u) \cap \{y_1, y_2, y_3\}, N_{H'}(v) \cap \{y_1, y_2, y_3\} \}$ is a partition of $\{y_1, y_2, y_3\}$.
Thus $|N_{H'}(u) \cap \{y_1, y_2, y_3\}| + | N_{H'}(v) \cap \{y_1, y_2, y_3\}| = 3$.
Without loss of generality, we may assume that $|N_{H'}(u) \cap \{y_1, y_2, y_3\}|=1$. Then $\{x_1, x_2, v, y_1, y_2,y_3, u\} \setminus (N_{H'}(u) \cap \{y_1, y_2, y_3\})$ forms $K_3 \vee I_3$ in $H'$ and we reach a contradiction.
\end{proof}

\begin{Thm}\label{thm:K_3,3-minor-free}
For any $(2,2)$ digraph $D$, if the underlying graph of $D$ is chordal, then the phylogeny graph of $D$ is $K_{3,3}$-minor-free.
\end{Thm}
\begin{proof}
Suppose, to the contrary, that $K_{3,3}$ is a minor of $P(D)$.
Then $K_{3, 3}$ is obtained from $P(D)$ by edge deletions or vertex deletions or contractions.
Let $(X,Y)$ be a bipartition of $K_{3,3}$.
Among the edge deletions, the vertex deletions, and the contractions to obtain $K_{3,3}$ from $P(D)$, we only take all the vertex deletions and all the contractions and apply them in the same order as the order in which vertex deletions and contractions applied to obtain $K_{3,3}$ from $P(D)$. Let $H^*$ be a graph obtained from $P(D)$ in this way.
Then $H^*$ contains $K_{3,3}$ as a spanning subgraph.
In addition, since $P(D)$ is chordal, $H^*$ is chordal by Lemma~\ref{lem:contraction} (it is clear that the chordality is preserved under vertex deletions).
If there is a pair of nonadjacent vertices in $H^*$ in each of $X$ and $Y$, then those four vertices form a hole in $H^*$ and we reach a contradiction.
Thus $X$ or $Y$ forms a clique in $H^*$ and so $H^*$ contains $K_3 \vee I_3$ as a spanning subgraph.
Then $K_3 \vee I_3$ is a minor of $P(D)$, which contradicts Lemma~\ref{lem:(K_3,I_3)-minor-free}.
Hence $P(D)$ is $K_{3,3}$-minor-free.
\end{proof}

\begin{Thm}\label{thm:planar}
For any $(2,2)$ digraphs, if its underlying graph is chordal, then its phylogeny graph is chordal and planar.
\end{Thm}
\begin{proof}
Let $D$ be a $(2,2)$ digraph whose underlying graph is chordal.
Then, by Theorem~\ref{thm:(2,j) chordal}, $P(D)$ is chordal.
Furthermore, by Corollary~\ref{cor:K_5-minor-free} and Theorem~\ref{thm:K_3,3-minor-free}, $P(D)$ is planar.
\end{proof}

\begin{Cor}
A chordal graph one of whose orientations is a $(2, 2)$ digraph is planar.
\end{Cor}
\begin{proof}
Let $G$ be a chordal graph one of whose orientations, namely $D$, is a $(2, 2)$ digraph.
Then $U(D)$ is $G$ which is chordal.
Thus, by Theorem~\ref{thm:planar}, $P(D)$ is planar.
Since $U(D)$ is a subgraph of $P(D)$, $U(D)$ is planar.
\end{proof}

\end{document}